\documentclass[12pt]{amsart}
\usepackage[colorlinks=true, pdfstartview=FitV, linkcolor=blue, citecolor=blue, urlcolor=blue]{hyperref}

\usepackage{amssymb,amsmath, amscd}
\usepackage{times, verbatim}
\usepackage{graphicx}
\usepackage[english]{babel}
 \usepackage[usenames, dvipsnames]{color}
\usepackage{amsmath,amssymb,amsfonts}
\usepackage{enumerate}
\usepackage{anysize}
\marginsize{3cm}{3cm}{3cm}{3cm}
\input xy
\xyoption{all}
\usepackage{pb-diagram}
\usepackage[all]{xy}
\input xy
\xyoption{all}

\DeclareFontFamily{OT1}{rsfs}{}
\DeclareFontShape{OT1}{rsfs}{n}{it}{<-> rsfs10}{}
\DeclareMathAlphabet{\mathscr}{OT1}{rsfs}{n}{it}

\begin{document}
\theoremstyle{plain}

\newtheorem{theorem}{Theorem}[section]
\newtheorem{thm}[equation]{Theorem}
\newtheorem{prop}[equation]{Proposition}
\newtheorem{cor}[equation]{Corollary}
\newtheorem{conj}[equation]{Conjecture}
\newtheorem{lemma}[equation]{Lemma}
\newtheorem{definition}[equation]{Definition}
\newtheorem{question}[equation]{Question}
\theoremstyle{definition}
\newtheorem{remark}[equation]{Remark}
\theoremstyle{definition}
\newtheorem{example}[equation]{Example}
\numberwithin{equation}{section}

\newcommand{\Hecke}{\mathcal{H}}
\newcommand{\Liea}{\mathfrak{a}}
\newcommand{\Cmg}{C_{\mathrm{mg}}}
\newcommand{\Cinftyumg}{C^{\infty}_{\mathrm{umg}}}
\newcommand{\Cfd}{C_{\mathrm{fd}}}
\newcommand{\Cinftyfd}{C^{\infty}_{\mathrm{ufd}}}
\newcommand{\sspace}{\Gamma \backslash G}
\newcommand{\PP}{\mathcal{P}}
\newcommand{\bfP}{\mathbf{P}}
\newcommand{\bfQ}{\mathbf{Q}}
\newcommand{\Siegel}{\mathfrak{S}}
\newcommand{\g}{\mathfrak{g}}
\newcommand{\A}{{\rm A}}
\newcommand{\B}{{\rm B}}
\newcommand{\Q}{\mathbb{Q}}
\newcommand{\Gm}{\mathbb{G}_m}
\newcommand{\kk}{\mathfrak{k}}
\newcommand{\nn}{\mathfrak{n}}
\newcommand{\tF}{\tilde{F}}
\newcommand{\p}{\mathfrak{p}}
\newcommand{\m}{\mathfrak{m}}
\newcommand{\bb}{\mathfrak{b}}
\newcommand{\Ad}{{\rm Ad}\,}
\newcommand{\ttt}{\mathfrak{t}}
\newcommand{\frakt}{\mathfrak{t}}
\newcommand{\U}{\mathcal{U}}
\newcommand{\Z}{\mathbb{Z}}
\newcommand{\bfG}{\mathbf{G}}
\newcommand{\bfT}{\mathbf{T}}
\newcommand{\R}{\mathbb{R}}
\newcommand{\ST}{\mathbb{S}}
\newcommand{\h}{\mathfrak{h}}
\newcommand{\bC}{\mathbb{C}}
\newcommand{\C}{\mathbb{C}}
\newcommand{\E}{\mathbb{E}}
\newcommand{\F}{\mathbb{F}}
\newcommand{\N}{\mathbb{N}}
\newcommand{\qH}{\mathbb {H}}
\newcommand{\temp}{{\rm temp}}
\newcommand{\Hom}{{\rm Hom}}
\newcommand{\ndeg}{{\rm ndeg}}
\newcommand{\Aut}{{\rm Aut}}
\newcommand{\Ext}{{\rm Ext}}
\newcommand{\Nm}{{\rm Nm}}
\newcommand{\End}{{\rm End}\,}
\newcommand{\Ind}{{\rm Ind}\,}
\def\circG{{\,^\circ G}}
\def\M{{\rm M}}
\def\diag{{\rm diag}}
\def\Ad{{\rm Ad}}
\def\G{{\rm G}}
\def\H{{\rm H}}
\def\SL{{\rm SL}}
\def\PSL{{\rm PSL}}
\def\GSp{{\rm GSp}}
\def\PGSp{{\rm PGSp}}
\def\Sp{{\rm Sp}}
\def\St{{\rm St}}
\def\GU{{\rm GU}}
\def\ind{{\rm ind}}
\def\SU{{\rm SU}}
\def\U{{\rm U}}
\def\GO{{\rm GO}}
\def\GL{{\rm GL}}
\def\PGL{{\rm PGL}}
\def\GSO{{\rm GSO}}
\def\Gal{{\rm Gal}}
\def\SO{{\rm SO}}
\def\O{{\rm  O}}
\def\Sym{{\rm Sym}}
\def\sym{{\rm sym}}
\def\St{{\rm St}}
\def\tr{{\rm tr\,}}
\def\ad{{\rm ad\, }}
\def\Ad{{\rm Ad\, }}
\def\rank{{\rm rank\,}}

\subjclass{Primary 11F70; Secondary 22E55}

\title[Multiplicities under basechange: finite field case]
{      Multiplicities under basechange: finite field case  }

\begin{abstract}
  A general proposition is proved relating multiplicities (for the restriction of a representation of a group to a subgroup) under basechange for groups over finite fields, and used to calculate certain multiplicities for cuspidal representations of general linear groups
  which become principal series representations under basechange for which the multiplicities can be calculated by geometric methods. For $\E/\F$ quadratic extension of finite fields, we use this method to calculate which 1 dimensional representations of $\GL_n(\E)$ appear in a cuspidal representation of $\GL_{2n}(\F)$, for which we also calculate which 1 dimensional representations of $\GL_n(\F) \times \GL_n(\F)$ appear in a cuspidal representation of $\GL_{2n}(\F)$.
  \end{abstract}

\author{Dipendra Prasad}
  \address{Indian Institute of Technology Bombay, Powai, Mumbai-400076} 
\address{Tata Institute of Fundamental
Research, Colaba, Mumbai-400005.}
\email{prasad.dipendra@gmail.com}
\maketitle
    {\hfill \today}
    
\tableofcontents

\section{Introduction}

Let $G$ and $H$ be two connected algebraic groups over a finite field $\F$ of order $q$ with $H \subset G$.
Let $\pi_1$ be an irreducible representation of $G(\F)$ and $\pi_2$ of $H(\F)$. Assume that
both the representations $\pi_1,\pi_2$ have a basechange, denoted as
$\pi_1^{\E}, \pi_2^\E$, to $\E$, where all through the paper, $\E$ is the unique quadratic extension of $\F$, with $\pi_1^{\E}$ an
irreducible representation of $G(\E)$ which is invariant under $ \langle \sigma \rangle = \Gal (\E/\F)$,  and $\pi_2^{\E}$ an irreducible representation of $H(\E)$ which is invariant under $ \langle \sigma \rangle = \Gal (\E/\F)$. In our paper, we will tacitly assume that $G,H$ are reductive algebraic groups over $\F$, and that $\pi_1,\pi_2$ are uniform representations, i.e., a virtual sum of Deligne-Lusztig representation $R(T,\theta)$ (for varying maximal tori $T \subset G$
and characters
$\theta: T(\F)\rightarrow \C^\times$). As an example, note that all irreducible representations of $\GL_n(\F)$ and of $\U_n(\F)$ are uniform representations. For such uniform representations, existence of basechange is a well-known theorem due to Digne-Michel \cite{DM}.

Our usage of the basechange depends mostly with the Shintani
character identity relating twisted character of $\pi^\E$ at $g \in G(\E)$ with the ordinary character of $\pi$
at the norm of the element $g$ which is a well-defined conjugacy class in $G(\F)$, denoted $\Nm(g)$.

The aim of this paper is to relate the multiplicity,
$$m(\pi_1,\pi_2)= \dim \Hom_{H(\F)}(\pi_1, \pi_2),$$
with the multiplicity of the basechanged representations:
$$m(\pi_1^{\E},\pi_2^{\E})= \dim \Hom_{H(\E)}(\pi_1^{\E}, \pi_2^{\E}).$$

It is well-known that basechange allows one to simplify representations, and therefore
allows one, in some cases, to calculate $m(\pi_1,\pi_2)$ from the simpler
information $m(\pi_1^{\E},\pi_2^{\E})$.

This paper itself is inspired by one such application where $G= \GL_{2n}(\F)$, and $H= \GL_{n}(\E)$ sitting naturally inside $G$, the representation $\pi_1$ being a cuspidal representation of $G= \GL_{2n}(\F)$, and $\pi_2 = \chi$, a one-dimensional representation of $H= \GL_{n}(\E)$ arising from a character
$\chi : \E^\times \rightarrow \C^\times$ through the determinant map
$\det:  \GL_{n}(\E) \rightarrow \E^\times$. (We use the same notation $\chi$ to denote a
character of $\E^\times$ as well as the associated character of $\GL_n(\E)$).
The question that we wish to understand is
the multiplicity $m(\pi_1,\chi)$.

Observe that over $\E$, $H(\E) \subset G(\E)$ is $\GL_n(\E) \times \GL_n(\E) \subset
\GL_{2n}(\E)$, and the representation $\pi_1^\E$  of $ \GL_{2n}(\E)$ becomes a principal series representation
$\pi \times \pi^\sigma$ induced from the $(n,n)$ parabolic of  $ \GL_{2n}(\E)$ with Levi subgroup
$\GL_n(\E) \times \GL_n(\E)$, where $\pi$ is the cuspidal representation of $\GL_n(\E)$ associated
to the same character $\theta: \F_{\!q^{2n}}^\times \rightarrow \C^\times$ which is used to define the
cuspidal representation $\pi_1$ of $G= \GL_{2n}(\F)$.

Basechange thus allows one to reduce a question on cuspidal representations to one on principal series representations which can be treated by `geometric' methods, as we show in this paper in one illustrative case.

\section {Multiplicity under basechange} \label{bc}

We keep the notation introduced in the introduction, thus $H \subset G$
are connected algebraic groups over a finite field $\F$,
$\pi_1$  an irreducible representation of $G(\F)$ and $\pi_2$ of $H(\F)$. We assume that
both the representations $\pi_1,\pi_2$ have a basechange, denoted as
$\pi_1^{\E}, \pi_2^\E$, to $\E$, with $\pi_1^{\E}$ an
irreducible representation of $G(\E)$
which is invariant under $ \langle \sigma \rangle = \Gal (\E/\F)$,  and $\pi_2^{\E}$ an irreducible representation of $H(\E)$ which is invariant under $ \langle \sigma \rangle = \Gal (\E/\F)$. We will fix an extension
$\tilde{\pi_1}^{\E}$  of
the irreducible representation  $\pi_1^{\E}$ of $G(\E)$ to
$G(\E) \rtimes \langle \sigma \rangle $; similarly, fix an extension $\tilde{\pi_2}^{\E}$  of
the irreducible representation  $\pi_2^{\E}$ of $H(\E)$ to
$H(\E) \rtimes \langle \sigma \rangle $. (In fact, the Shintani character identity fixes a unique extension of
$\pi_1^\E$ and $\pi_2^\E$ to the representation $\tilde{\pi_1}^{\E}$  of
$G(\E) \rtimes \langle \sigma \rangle $ and to the representation $\tilde{\pi_2}^{\E}$ of
$H(\E) \rtimes \langle \sigma \rangle $.)

The following proposition, much more general than Theorem 1 in \cite{Pr},
has the same proof as there. 

\begin{prop} \label{mult} With the notation as above, we have:
$$2 m(\tilde{\pi_1}^{\E}, \tilde{\pi_2}^{\E}) =  m({\pi_1}^{\E}, {\pi_2}^{\E}) +  m({\pi_1}, {\pi_2}).$$
  \end{prop}
\begin{proof}
  Recall that by the Schur orthogonality theorem, if $V$ is a representation of a finite group
  ${\mathcal H}$ with character $\Theta$, then
  $$\dim V^{\mathcal H} = \frac{1}{|{\mathcal H}|} \sum_{h \in {\mathcal H}} \Theta(h).$$

  We will apply the Schur orthogonality theorem to the representation $V$
of the group ${\mathcal H}= H(\E) \rtimes \langle \sigma \rangle $
    which is the restriction of the representation
  $(\tilde{\pi_1}^{\E})^\vee \otimes \tilde{\pi_2}^{\E}$ of the group 
  $ [G(\E) \rtimes \langle \sigma \rangle] \times [H(\E) \rtimes \langle \sigma \rangle]$ to the diagonally embedded subgroup ${\mathcal H}= H(\E) \rtimes \langle \sigma \rangle $.

  Note that  ${\mathcal H}= H(\E) \rtimes \langle \sigma \rangle = H(\E) \cup
  H(\E) \cdot \sigma  $, hence the sum of characters on ${\mathcal H}$ decomposes as a sum over
  $H(\E)$ and another sum over $ H(\E) \cdot \sigma  $. By the Shintani character identity
  which involves the norm mapping
  $$\Nm:  G(\E) \times H(\E) \rightarrow  G(\F) \times H(\F),$$
  the character of $(\tilde{\pi_1}^{\E})^\vee \otimes \tilde{\pi_2}^{\E}$
  at $ (g,h) \cdot \sigma \in [G(\E) \times H(\E)] \cdot \sigma$ is the character
  of the representation $({\pi_1})^\vee \otimes {\pi_2}$ at the element
  $ (\Nm g, \Nm h) \in G(\F) \times H(\F) $. The mapping $\Nm:  G(\E) \times H(\E) \rightarrow  G(\F) \times H(\F)$ has the well-known property that the cardinality of
  $\sigma$-centralizer of an element $ (g,h) \in G(\E) \times H(\E)$ is the same
  as the cardinality of the centralizer of  $ (\Nm g, \Nm h)$ in $G(\F) \times H(\F) $
  (cf. Lemma 2 in \cite{Pr}),
  allowing one to conclude the proposition. \end{proof}

\begin{remark}
  Observe that since  $m(\tilde{\pi_1}^{\E}, \tilde{\pi_2}^{\E})$ is the dimension of the space of
  ${\mathcal H}= H(\E) \rtimes \langle \sigma \rangle $ invariant linear vectors in the space
 $\Hom_\C(\tilde{\pi_1}^{\E}, \tilde{\pi_2}^{\E})$, and 
  $ m({\pi_1}^{\E}, {\pi_2}^{\E})$ is the dimension of the space of
  $H(\E) $ invariant linear vectors in the same space,
  $$0 \leq m(\tilde{\pi_1}^{\E}, \tilde{\pi_2}^{\E}) \leq m({\pi_1}^{\E}, {\pi_2}^{\E}).$$
  \end{remark}

\begin{cor} \label{mult-cor} With the notation as above,
  $$m({\pi_1}, {\pi_2}) \leq m({\pi_1}^{\E}, {\pi_2}^{\E}),$$
  and $$m({\pi_1}, {\pi_2}) \equiv m({\pi_1}^{\E}, {\pi_2}^{\E}) \bmod 2.$$
  In particular, if  $m({\pi_1}^{\E}, {\pi_2}^{\E}) \leq 1$, then $m({\pi_1}, {\pi_2}) \leq 1$, and
$$m({\pi_1}^{\E}, {\pi_2}^{\E}) = m({\pi_1}, {\pi_2}).$$
\end{cor}

The following corollary of Proposition \ref{mult} is Theorem 1 in \cite{Pr}.

\begin{cor} Let $G$ be a connected reductive group over a finite field $\F$, $\E/\F$ a quadratic extension,
  and $\pi$ an irreducible uniform 
representation of $G(\E)$, i.e., one  which can be expressed as a sum of Deligne-Lusztig 
representations of $G(\E)$ induced from tori.  Then the representation $\pi$ has a 
$G(\F)$ fixed vector if and only if $\pi^\sigma \cong \pi^\vee$. If $\pi^\sigma \cong \pi^\vee$, then $\pi$ has a one dimensional space of fixed vectors
under $G(\F)$, and the representation $\pi \otimes \pi^\sigma$ which is canonically a representation of $[G(\E)  \times G(\E)] \rtimes \Z/2$ 
has a $G(\E) \rtimes \Z/2$ fixed vector. \end{cor}

\begin{example}
  Here is a simple example of the way multiplicities vary under basechange as dictated by
  Proposition \ref{mult}. If $P(\chi) = P(\chi^{-1})$
  denotes the irreducible principal series representation
  of $\PGL_2(\F)$ associated to a character $\chi: \F^\times \rightarrow \C^\times$, $\chi^2 \not = 1$,
  then it is easy to see that for any  three  characters
  $\chi_1,\chi_2,\chi_3: \F^\times \rightarrow \C^\times$, with $\chi_i^2 \not = 1$, and $P(\chi_i)$ distinct,
  $m(P(\chi_1) \otimes P(\chi_2), P(\chi_3)) =1$ except when
  $\chi_1^{\pm 1} \chi_2^{\pm 1} \chi_3^{\pm 1} = 1$, 
  and that in these exceptional cases,  $m(P(\chi_1) \otimes P(\chi_2), P(\chi_3))
  =2$. On the other hand,  if $D(\chi) = D(\chi^{-1})$  
  denotes the irreducible cuspidal representation
  of $\PGL_2(\F)$ associated to a character $\chi: \E^\times/\F^\times \rightarrow \C^\times$,
  $\chi^2 \not = 1$,
  then it is easy to see that for any 3  characters $\chi_1,\chi_2,\chi_3:
  \E^\times/\F^\times \rightarrow \C^\times$, with $\chi_i^2 \not = 1$, and $D(\chi_i)$ distinct,
  $m(D(\chi_1) \otimes D(\chi_2), D(\chi_3)) =1$
  except when
$\chi_1^{\pm 1} \chi_2^{\pm 1} \chi_3^{\pm 1} = 1$,
  and that in these exceptional cases,  $m(D(\chi_1) \otimes D(\chi_2), D(\chi_3)) =0$.
  \end{example}
\section{Linear periods for Principal series} \label{linear-ps}

Following is the main proposition of the paper proved by geometric means.

\begin{prop} \label{princ}
  Let $\pi_1,\pi_2$ be two irreducible cuspidal representations of $\GL_n(\F)$ and $\chi_1,\chi_2$ be two characters of $\, \F^\times$.
  Then the representation $\pi_1 \times \pi_2$ of $\GL_{2n}(\F)$
  (parabolically induced from the $(n,n)$-parabolic)
  has a nonzero vector on which $\GL_n(\F) \times \GL_n(\F)$
  operates by the character $\chi_1 \times \chi_2$
   if and only if one of the two conditions hold:
\begin{enumerate}
 \item $(\pi_1 \otimes \chi_1^{-1})^\vee \cong  \pi_2 \otimes \chi_2^{-1}.$
 \item $n=2m$ is even, $\pi_1$ contains a vector on which  $\GL_m(\F) \times \GL_m(\F)$
   operates via $\chi_1 \times \chi_2$ and $\pi_2$ contains a vector on which  $\GL_m(\F) \times \GL_m(\F)$
   operates via $\chi_1 \times \chi_2$.
   \end{enumerate}
  If $n>1$, then the dimension of the space of vectors in $\pi_1 \times \pi_2$ on which $\GL_n(\F) \times \GL_n(\F)$
  operates by the character $\chi_1 \times \chi_2$ is the sum of dimensions arising from these two options, first one when
  it holds contributes dimension 1, and the second one contributes $m(\pi_1, \chi_1 \times \chi_2) \cdot m(\pi_2, \chi_1 \times \chi_2)$.
      \end{prop}

\begin{proof}
  We begin the proof by first taking care of the case $n=1$. Thus, suppose $\pi_1= \lambda_1$ and $\pi_2= \lambda_2$ are
  two characters of $F^\times$, and $\lambda_1 \times \lambda_2$ is the corresponding principal series
  representation of $\GL_2(F)$. As is well-known, the character of the principal series
  representation $\lambda_1 \times \lambda_2$ of $\GL_2(F)$ 
  at the diagonal element $(x,y)$ of  $\GL_2(F)$
  is equal to $\lambda_1(x)\lambda_2(y) + \lambda_2(x)\lambda_1(y)$ (for $x \not = y$), and is equal to $(q+1) \lambda_1(x) \lambda_2(x)$ at the central elements $(x,x)$. It follows that  the principal series
  representation $\lambda_1 \times \lambda_2$ (assumed to be irreducible, equivalently, $\lambda_1 \not =  \lambda_2$)
  of $\GL_2(F)$ contains all characters $\mu_1 \times \mu_2$ of
  $F^\times \times F^\times$ with $\mu_1 \mu_2 = \lambda_1 \lambda_2$ with multiplicity 1 except for the characters  $\lambda_1 \times \lambda_2$  and $\lambda_2 \times \lambda_1$ which appear with multiplicity 2. This proves the proposition
  in the case $n=1$.

  In the rest of the proof we will assume that $n>1$, pointing out exactly when this hypothesis is first used.
  We will calculate the dimension of the space of vectors in $\pi_1 \times \pi_2$ on
  which ${\H}= \GL_n(\F) \times \GL_n(\F)$
  operates by the character $\chi_1 \times \chi_2$ 
    by a direct application of the Mackey theory. Recall that Mackey theory
gives an answer using the double coset decomposition
$${\H} \backslash \GL_{2n}(\F) / P$$
where $P $ is the $(n,n)$ parabolic in $\GL_{2n}(\F)$.

Suppose $V = V_1 \oplus V_2$ is the decomposition of a $2n$-dimensional vector space over $\F$
as a direct sum of two $n$-dimensional subspaces, which realizes
${\H}= \GL_n(\F) \times \GL_n(\F)$ as $\H=\GL(V_1) \times \GL(V_2)$.
Let $W$ be an $n$-dimensional subspace
of $V$ whose  stabilizer $P=P(W)$ defines a  parabolic subgroup of $\GL(V)$.

Clearly,  the double cosets in ${\H} \backslash \GL(V) / P(W)$ are in bijective correspondence with $\GL(V)$-conjugacy classes of
triples $(V_1,V_2,W)$  of subspaces of dimension $n$  in $V$ with $V_1+V_2=V$.  
From this it is easy to see that the double cosets in ${\H} \backslash \GL(V) / P$ are
parametrized by pairs of integers $(r,s)$ with $0\leq r,s \leq n$ with the only constraint that $r+s \leq n$. The pair $ (r,s)$ corresponds to the pair
$(\dim(W\cap V_1), \dim (W\cap V_2))$.

To make a detailed calculation, let $V_1,V_2,W,W'$ be subspaces, each of dimension $n$,
of a vector space $V$ of dimension $2n$, with the following basis vectors for the subspaces
 $V_1,V_2,W,W'$ for integers $r \geq 0,s \geq 0,t \geq 0$ with $r+s+t=n$, and $W\oplus W'=V$:
\begin{eqnarray*}
  V_1 & = & \{e_1,e_2,\cdots, e_r; g_1,\cdots, g_t; v_1,\cdots, v_s\}  ,\\
  V_2 & = & \{f_1,f_2,\cdots, f_s; h_1,\cdots, h_t; w_1,\cdots, w_r\} ,\\
  W & = & \{e_1,\cdots, e_r; f_1,\cdots, f_s; g_1+h_1, \cdots, g_t+h_t\}, \\
  W' & = & \{g_1,\cdots, g_t; v_1,\cdots, v_s; w_1,\cdots, w_r\}.
  \end{eqnarray*}
We have,
\begin{eqnarray*}
  \dim(V_1 \cap W) & = & r ,\\
  \dim(V_2 \cap W) & = & s ,
\end{eqnarray*}
which, since $V_1+V_2=V$, implies that,
  \begin{eqnarray*}
  \dim(V_1 \cap [V_2+ W]) & = & n-s = r+t, \\
  \dim(V_2 \cap [V_1 + W]) & = & n-r= s+t.
\end{eqnarray*}

  To apply the Mackey theory, we need to calculate $\A= [\GL(V_1) \times \GL(V_2)] \cap P(W)$, and its projection $\B$ to $P(W)/U(W)$ where $U(W)$ is the unipotent radical of $P(W)$, so that this double coset gives the representation,

  $${\rm Ind}_{\A} ^{[\GL(V_1) \times \GL(V_2)]} \rho,$$
    where $\rho$ is the representation of the subgroup $\A$  of $\GL(V_1) \times \GL(V_2)$
    operating through the restriction to $\B$ of the representation $\pi_1 \times \pi_2$
    of $P(W)/U(W) = \GL(W) \times \GL(W')$.

    We now calculate $\A,\B$ with      $$\A= [\GL(V_1) \times \GL(V_2)] \cap P(W) = \left \{g\in \GL(V)| g(V_1)=V_1, g(V_2)=V_2, g(W)=W \right \}.$$
    Note that an element $g \in \A$ leaves  $(V_1 \cap W), (V_2 \cap W), (V_1 \cap [V_2+ W]),
    (V_2 \cap [V_1 + W])$     invariant.

    To understand  the subgroup $\B$, we will write an element $g \in \A \subset
    \GL(V)$ in the basis of $V = W\oplus W'$ afforded by concatenation of the basis for $W,W'$ which we recall has the following basis,
\begin{eqnarray*}
  W & = & \{e_1,\cdots, e_r; f_1,\cdots, f_s; g_1+h_1, \cdots, g_t+h_t\}, \\
  W' & = & \{g_1,\cdots, g_t; v_1,\cdots, v_s; w_1,\cdots, w_r\},
  \end{eqnarray*}
    in the form:
$$g = \left ( \begin{array}{cccccr} 
  * & 0  & B  & B  & * & 0  \\
0 & * & * & 0 & 0 & * \\
0 & 0 & A & 0 & 0 & -C \\
0 & 0 & 0 & A & * & C \\
0&0  & 0 & 0 & *&0 \\
0& 0& 0& 0& 0& * 
\end{array}
\right ) \in \GL(V),$$
where each entry corresponds to a block matrix, for example, the entry at place $(1,1)$ corresponds to the endomorphism of the subspace $V_1 \cap W = \{e_1,\cdots, e_r\}$; all the entries denoted by $A,B,C, *$ are arbitrary matrices of appropriate sizes. 

It follows that in $P(W)/U(W)= \GL(W) \times \GL(W')$, $g$ looks like,

$$g = \left ( \begin{array}{cccccc} 
  * & 0  & B  &   &  &   \\
0 & * & * &  &  &  \\
0 & 0 & A &  &  &  \\
 &  &  & A & * & C \\
&  &  & 0 & *&0 \\
& & & 0& 0& * 
\end{array}
\right ) \in \GL(W) \times \GL(W'),$$
and once again, all the nonzero entries in the matrix are arbitrary (and invertible if necessary).

If
$0< r+s<n$, the above subgroup of matrices contains the unipotent subgroup of a nontrivial
parabolic in
$\GL(W) \times \GL(W')$. Therefore since we are dealing with parabolic induction arising
from a cuspidal data, the double cosets represented by $(r,s)$ with $0< r+s<n$, do not carry any vector left invariant (up to a character) by the subgroup  ${\H} = \GL_n(\F) \times \GL_n(\F)$.

If $r=0, s=0$, then
the matrix $g$ above simplifies to, 
$$g = \left ( \begin{array}{cccc} 
  \alpha & 0      \\
0 & \alpha &  
\end{array}
\right ) \in \GL(W) \times \GL(W'),$$
where $\alpha$ represent an arbitrary matrix in $\GL_n(\F)$.
In this case, ${\rm Ind}_{\A} ^{[\GL(V_1) \times \GL(V_2)]} (\pi_1 \otimes \pi_2)|_{\A},$
has a nonzero vector on which $\GL(V_1) \times \GL(V_2)$ acts by the character
$\chi_1 \times \chi_2$ if and only if

$$(\pi_1 \otimes \chi_1^{-1})^\vee \cong  \pi_2 \otimes \chi_2^{-1}.$$

If
$r+s=n$,
the matrix $g$ above simplifies to,

$$g = \left ( \begin{array}{cccc} 
  *_r & 0  &  &     \\
0 & *_s &  &    \\
 &  & *_s & 0  \\
 &  & 0 & *_r  \\
\end{array}
\right ) \in \GL(W) \times \GL(W'),$$
where $*_r,$ respectively $*_s$, represent an arbitrary matrix in $\GL_r(\F)$, respectively in $\GL_s(\F)$.

Appeal now to the well-known result that a cuspidal representation $\pi$
of $\GL_n(\F), n> 1$ has a nonzero vector on which the Levi subgroup $\GL_r(\F) \times \GL_s(\F)$ for $r+s=n$ acts by a character only if $r=s$, in particular $n$ must be even, see for example Proposition 2.14 of
\cite{Se}; it is here that we need $n>1$.

The proof of the proposition in case $n> 1$ is completed after having observed that a block diagonal matrix
$$g = \left ( \begin{array}{cccc} 
  A &   &  &     \\
 & B &  &    \\
 &  & C &   \\
 &  &  & D  \\
\end{array}
\right ) \in \H= \GL_n(\F) \times \GL_n(\F),$$
(where each of the matrices $A,B,C,D$ are of size $n/2$),  when considered as an element of
$\GL(W) \times \GL(W'),$ looks like:
$$g = \left ( \begin{array}{cccc} 
  A &   &  &     \\
 & C &  &    \\
 &  & B &   \\
 &  &  & D  \\
\end{array}
\right ) .$$
  \end{proof} 

\section{Linear periods for cuspidal representations} \label{linear-cusp}

Let $M= \GL_n(\F) \times \GL_n(\F) $ be the standard Levi subgroup inside $\GL_{2n}(\F)$. The aim of this section is to
understand the characters of $M$ which appear in a cuspidal representation of $\GL_{2n}(\F)$. Here is the main theorem
of this section which gives a complete understanding of this question.

\begin{thm} \label{linear-period}
  Let $\pi = \pi(\theta)$ be a cuspidal representation of $\GL_{2n}(\F)$ associated to a character
  $\theta: \F^\times_{q^{2n}} \rightarrow \C^\times$. Let $\chi_1,\chi_2$ be a pair of characters of $\F^\times$
  identified to a pair of characters of $\GL_n(\F)$ through the determinant map, and hence a character
  $\chi_1 \times \chi_2: M= \GL_n(\F) \times \GL_n(\F) \rightarrow \C^\times$. Then the character $\chi_1 \times \chi_2$
  appears in $\pi$ if and only if $(\chi_1\cdot \chi_2)(\Nm x) = \theta(x)$ for all elements
  $x \in   \F^\times_{q^{n}}  \subset \F^\times_{q^{2n}}$ where $\Nm x$ is the norm mapping
  $\Nm:  \F^\times_{q^{n}}  \rightarrow \F^\times_{q};$ in particular, the cuspidal representation $\pi$
  of $\GL_{2n}(\F)$ contains a character of $M$ if and only if $\theta$ restricted to $ \F^\times_{q^{n}}$ factors through
  the norm mapping  ($\Nm:  \F^\times_{q^{n}}  \rightarrow \F^\times_{q}$). Further,
  $$m(\pi, \chi_1 \times \chi_2) \leq 1.$$
  \end{thm} 

The proof of this theorem will depend on the following theorem of the author from \cite{Pr2} which is Theorem 1 there.

\begin{thm} \label{deg-whi}Let $\pi = \pi(\theta)$ be a cuspidal representation of $\GL_{2n}(\F)$ associated to a character
  $\theta: \F^\times_{q^{2n}} \rightarrow \C^\times$. Let $P$ be the $(n,n)$ parabolic in  $\GL_{2n}(\F)$ with $M= \GL_n(\F) \times \GL_n(\F) $ its standard Levi subgroup, and $N = \M_n(\F)$ its unipotent radical. Let $\psi_0: \F \rightarrow \C^\times$
  be a fixed non-trivial character on $\F$, and $\psi: \M_n(\F) \rightarrow \C^\times$ be the character $\psi(X) = \psi_0(\tr X)$. Then $\pi_{N,\psi}$, the largest subspace of $\pi$ on which $N$ operates by $\psi$, is a representation space for
$\Delta (\GL_n(\F)) \subset \GL_n(\F) \times \GL_n(\F) $, and as a $\GL_n(\F)$-module  given by:
  $$\pi_{N,\psi} = \ind_{  \F^\times_{q^{n}}}^{\GL_n(\F)} (\theta|_{  \F^\times_{q^{n}}}). $$

  \end{thm}

We note the following corollary of this theorem as a simple consequence of the Clifford theory applied to the
group $P=MN$ containing the abelian normal subgroup $N$.

\begin{cor} \label{non-deg-cor}
 Let $\pi = \pi(\theta)$ be a cuspidal representation of $\GL_{2n}(\F)$ associated to a character
 $\theta: \F^\times_{q^{2n}} \rightarrow \C^\times$.  Let $P$ be the $(n,n)$ parabolic in  $\GL_{2n}(\F)$ with $M= \GL_n(\F) \times \GL_n(\F) $ its standard Levi subgroup, and $N = \M_n(\F)$ its unipotent radical. Then the subspace $\pi^{\ndeg}$  of $\pi$
 on which $N$ operates only by non-degenerate characters of $N$ is a module for $M= \GL_n(\F) \times \GL_n(\F) $ given by
 $$ \pi^{\ndeg} = \ind_{  \F^\times_{q^{n}}}^{\GL_n(\F) \times \GL_n(\F)} (\theta|_{  \F^\times_{q^{n}}}), $$
 where $ \F^\times_{q^{n}}$ is the diagonally embedded torus inside  $\GL_n(\F) \times \GL_n(F) $.
\end{cor}

We now return to the proof of Theorem \ref{linear-period}.

\begin{proof}
Let $\pi^{\deg}$ be the subspace of 
  $\pi$
on which $N$ operates by degenerate characters, thus:
$$ \pi = \pi^{\ndeg} \oplus \pi^{\deg}.$$

Clearly, $\pi^{\deg}$ is a module for $M=\GL_n(\F) \times \GL_n(\F)$.
  For the proof of Theorem \ref{linear-period}, it suffices by Corollary \ref{non-deg-cor} to prove that 
  $\pi^{\deg}$ as a module for $M=\GL_n(\F) \times \GL_n(\F)$ contains no characters of $M$.

  If $\lambda: N=\M_n(\F) \rightarrow \C^\times$ is any degenerate character of $\M_n(\F)$,
  let $M_\lambda$ be the
  subgroup of $M$ stabilizing the character $\lambda$. If $\pi_\lambda \subset \pi$ is the representation of $M_\lambda$
  consisting of those vectors in $\pi$ on which $N$ operates by  $\lambda: N=\M_n(\F) \rightarrow \C^\times$, then by
  Clifford theory,
  $$  \ind_{  M_\lambda}^{\GL_n(\F) \times \GL_n(\F)} (\pi_\lambda) $$
  is a submodule for $\pi$, and $\pi^{\deg}$ is a sum of these submodules.

  Thus it suffices to prove that none of the representations $  \ind_{  M_\lambda}^{\GL_n(\F) \times \GL_n(\F)} (\pi_\lambda) $
  contain a 1 dimensional representation, say $\chi$, of ${\GL_n(\F) \times \GL_n(\F)}$.
  By Frobenius reciprocity,
  $$\Hom_{\GL_n(\F) \times \GL_n(\F)}[\ind_{  M_\lambda}^{\GL_n(\F) \times \GL_n(\F)} (\pi_\lambda) , \chi] \cong
  \Hom_{ M_\lambda} [\pi_\lambda , \chi] .$$

  We next identify the set of degenerate characters $\lambda: \M_n(\F) \rightarrow \C^\times$ together with their stabilizers $M_\lambda
\subset \GL_n(\F) \times \GL_n(\F)$.

  Any character of $\M_n(\F)$
  is of the form $$\psi_A= \{X\rightarrow \psi({\rm tr}[AX]) | X \in \M_n(\F) \},$$
  where $\psi: \F\rightarrow \C^\times$ is a fixed non-trivial
character. Thus, it is seen that, up to the natural action of ${\GL_n(\F) \times \GL_n(\F)}$ on $\M_n(\F)$,
characters of $\M_n(\F)$ are represented by the rank of the matrix $A$ which is any intger $0\leq i \leq n$, with
$i=n$ representing non-degenerate characters. We will use as representatives of orbits of degenerate characters of $\M_n(\F)$ 
under the action of ${\GL_n(\F) \times \GL_n(\F)}$ on $\M_n(\F)$,
the matrices:

$$A_i = \left ( \begin{array}{cccc} 
  I_i & 0      \\
0 & 0 &  
\end{array}
\right ) \in \M_n(\F),$$
where $I_i$ represents the identity matrix in $\GL_i(\F)$.

Under the  action of ${\GL_n(\F) \times \GL_n(\F)}$ on $\M_n(\F)$
given by $(g_1,g_2) \cdot X = g_1Xg^{-1}_2$,
the associated action of ${\GL_n(\F) \times \GL_n(\F)}$ on the characters $\psi_A$, $A \in \M_n(\F)$
is given by $(g_1,g_2) \cdot \psi_A = \psi_{g^{-1}_2Ag_1}$. (The important point to notice here is the slight difference
in the action of ${\GL_n(\F) \times \GL_n(\F)}$ on $\M_n(\F)$ and on the character group of $\M_n(\F)$ which has been identified to $\M_n(\F)$; this
difference plays an important role below!)

It follows that under the action of ${\GL_n(\F) \times \GL_n(\F)}$
on the character $\lambda= \psi_{A_i}$, $A_i \in \M_n(\F)$ as above,   $M_\lambda \subset \GL_n(\F) \times \GL_n(\F)$
contains
$(1,g_2)$ with
$$g_2 = \left ( \begin{array}{cccc} 
  I_i & *      \\
0 & I_{n-i} &  
\end{array}
\right ) \in \GL_n(\F),$$
where $*$ represents arbitrary $i \times (n-i)$ matrices over $\F$.

Now observe that if   $\Hom_{ M_\lambda} [\pi_\lambda , \chi] \not = 0 $ for some character
 $\chi$ of ${\GL_n(\F) \times \GL_n(\F)}$,
then since the character $\chi$ of ${\GL_n(\F) \times \GL_n(\F)}$
  must be trivial on any unipotent subgroup of  ${\GL_n(\F) \times \GL_n(\F)}$,
  in our case $\lambda = \psi_{A_i}$, the character $\chi$ of $ M_\lambda$
  must be trivial on
  $(1,g_2) \in {\GL_n(\F) \times \GL_n(\F)}$, with
$$g_2 = \left ( \begin{array}{cccc} 
  I_i & *      \\
0 & I_{n-i} &  
\end{array}
\right ) \in \GL_n(\F),$$
where $*$ represents arbitrary $i \times (n-i)$ matrices over $\F$.
  
Since the character $\lambda = \psi_{A_i}: \M_n(\F) \rightarrow \C^\times$ itself is trivial on matrices of the form:
$$m = \left ( \begin{array}{cccc} 
 0 & *      \\
0 & * &  
\end{array}
\right ) \in \M_n(\F),$$
we find that $\pi$ (which contains $\pi_\lambda$ on which $N=\M_n(\F)$ operates by $\lambda$, and $M_\lambda$ which is contained in $ {\GL_n(\F) \times \GL_n(\F)}$ operates via $\chi$) contains a nonzero vector on which the unipotent group
of matrices:
$$ \left ( \begin{array}{ccccr} 
  I_i & 0   & 0  & *  \\
0  & I_{n-i}  & 0 & *  \\
 &  &   I_i & *  \\
& &  0 & I_{n-i}

\end{array}
\right ), $$
acts trivially.
Since this group of unipotent matrices is the unipotent radical of
the $(n+i, n-i)$ parabolic (with $i < n$), and the representation $\pi$ is given to be cuspidal, there can be no such nonzero vectors,
 completing  the proof of the theorem.
\end{proof}

\section{Twisted-linear periods for cuspidal representations} 

In this section we use the method of basechange developed in section 2 to study multiplicity questions
on one dimensional representations of $\GL_n(\E)$ inside a cuspidal representation of $\GL_{2n}(\F)$. (Since 
$\GL_n(\F) \times \GL_n(\F) \subset \GL_{2n}(\F)$ is often referred to as the `linear periods' case, we could call
$\GL_n(\E)  \subset \GL_{2n}(\F)$  as the `twisted-linear periods' case.)
For this,  the result obtained for principal series in section \ref{linear-ps}, and the result
obtained 
for cuspidal representations in section \ref{linear-cusp} will be used. We begin with an important result
which allows one to say that out of the two terms in Proposition \ref{princ} giving rise to multiplicity there,
only one term contributes.

  \begin{prop}
  \label{mult1}
  Let $\pi_1$ be an irreducible cuspidal representation of $\GL_{n}(\E)$
  where  as always, $\E$ is a quadratic extension of $\F$ with the non-trivial
  Galois automorphism $\sigma$. If,
   \[  \pi_1 ^\vee \cong  \pi_1 ^\sigma, \]
   then $n$ must be odd. If $n$ is odd, then an irreducible cuspidal representation $\pi_1 = \pi_1(\theta)$
   of $\GL_{n}(\E)$ associated to a character $\theta: \F_{q^{2n}}^\times \rightarrow \C^\times$ has the property
   $\pi_1 ^\vee \cong  \pi_1 ^\sigma,$ if and only if $\theta$ restricted   to $\F_{q^{n}}^\times$ is trivial.
\end{prop}
\begin{proof}
We analyze the condition $\pi_1 ^\vee \cong  \pi_1 ^\sigma,$ 
  using the fact that  two cuspidal representations $\pi_1(\theta_1)$ and $\pi_2(\theta_2)$ of $\GL_n(\E)$ arising out
  of characters $\theta_1,\theta_2: \F_{q^{2n}}^\times \rightarrow  \C^\times$ are isomorphic if and only if
  for some
  $\tau \in \Gal( \F_{q^{2n}}/\E)$, $\tau(\theta_1)=\theta_2$. We have been using $\sigma$ for the nontrivial element of
  $\Gal(\E/\F)$, but now we will also use $\sigma$ to denote any element of   $\Gal(\F_{q^{2n}}/\F)$ which projects to
  this nontrivial element of $\Gal(\E/\F)$.
  The isomorphism 
$\pi_1 ^\vee \cong  \pi_1 ^\sigma,$ 
  implies that for some
  $\tau \in \Gal( \F_{q^{2n}}/\E)$:
  \[ \tag{1} (\theta  \cdot \theta^{ \tau \sigma}) = 1. \]

  Applying $\tau \sigma$ to the equality in equation (1), we have:
 \[ \tag{2} (\theta^{\tau \sigma}  \cdot \theta^{ \tau^2 \sigma^2}) = 1. \]
 The two equations (1) and (2)  give:
   \[ \tag{3} \theta  =  \theta^{ \tau^2 \sigma^2}. \]

 Since $\theta$ gives rise to a cuspidal representation of $ \GL_{n}(\E)$, all its Galois conjugate under
 $\Gal( \F_{q^{2n}}/\E)$,
 are
  distinct, and therefore equation $(3)$ 
  implies that,   \[ \tag{4} (\tau \sigma)^2=1. \]

  Notice that $\tau \sigma \in \Gal( \F_{q^{2n}}/\F) = \Z/(2n)$ which projects to the nontrivial element of
  $ \Gal( \E/\F) =\Z/2$. Therefore, if $(\tau \sigma)^2=1$, it follows by group theoretic structure
  of $\Z/(2n)$, that $n$ must be an odd integer, i.e., the condition $\pi_1 ^\vee \cong  \pi_1^\sigma$ forces $n$ to be odd.

  Assuming $n$  odd, we now classify cuspidal representations $\pi_1=\pi_1(\theta)$ for which
  $\pi_1 ^\vee \cong  \pi_1^\sigma$ as in the statement of the proposition.
  
 Let $\sigma$  be the unique element of   $\Gal(\F_{q^{2n}}/\F)$ of order 2 (which automatically
  projects to
  the nontrivial element of $\Gal(\E/\F)$) whose fixed field defines $\F_{q^n}$. Since $n$ is odd, $\tau$ has odd order, and therefore by equation $(4)$, $\tau =1$.
  In this case, putting $\tau=1$ in the equality in equation $(1)$,
  we find   that
  $\theta$ restricted to $\F_{q^n}^{\times}$ is trivial,
  and conversely, if this is so, it is easy to see that
  $\pi_1 ^\vee \cong  \pi_1 ^\sigma$ 
  completing the proof of the proposition. 
\end{proof}

\begin{remark}
  The Proposition \ref{mult1} can be interpreted as implying that the image of the basechange map from $\U_n(\F)$ to $\GL_n(\E)$ contains a cuspidal representation if and only if $n$ is odd, a well-known conclusion.
  \end{remark}

The following theorem can be considered as  a contribution towards depth-zero
case of Conjecture 1 in \cite{PT}.

\begin{thm}\label{conj-PT}
  Let $\pi = \pi(\theta)$ be an irreducible cuspidal representation of $G=\GL_{2n}(\F)$, $n>1$.
  Assume that $\pi$ arises  from a character
  $\theta: \F_{q^{2n}}^\times \rightarrow \C^\times$.
Let $\E$ be a quadratic extension of $\F$, and $\chi : \E^\times \rightarrow \C^\times$, a character, thought of as a character of
 $H= \GL_{n}(\E)$
through the determinant map
$\det:  \GL_{n}(\E) \rightarrow \E^\times$.
Then the representation $\pi$ of $\GL_{2n}(\F)$ 
has a nonzero vector on which $\GL_n(\E)$
  operates by $\chi$  if and only if
  $\theta$ restricted to  $\F_{q^{n}}^\times$ arises from $\chi$ restricted to
    $\F^\times$ through the norm mapping   $\F_{q^{n}}^\times \rightarrow \F^\times$.
    In particular the condition that the character $\chi \circ \det$ of $\GL_n(\E)$, for
  $\chi: \E^\times \rightarrow  \C^\times$,  appears in $\pi$ depends only on $\chi$ restricted
  to $\F^\times$.  

  The dimension of the space of linear forms $m(\pi,\chi)$ when nonzero is 1.
\end{thm}
\begin{proof}
  The proof of this theorem  will be based on
  Proposition \ref{mult} relating multiplicities under basechange,
  the calculation done in Proposition \ref{princ} regarding the multiplicity
  for a principal series representation, and Theorem \ref{linear-period} which calculates multiplicities for a
  cuspidal representation.

Write, as before, the representation
$ \pi^{\E} $
of $\GL_{2n}(\E)$ as  $\pi^\E =\pi_1 \times \pi_1^\sigma$
where $\pi_1$ is a cuspidal representation of $\GL_n(\E)$, arising from the same character
  $\theta: \F_{q^{2n}}^\times \rightarrow \C^\times$,
 and $\sigma$ is the Galois action on $\GL_n(\E)$.
 For $\chi: \E^\times \rightarrow \C^\times$, let $\chi^\E= \chi \times \chi^\sigma: \E^\times \times
  \E^\times \rightarrow \C^\times$.

 If $n$ is an odd integer with 
$m({\pi}^{\E}, {\chi}^{\E}) \not = 0$, then by Proposition \ref{princ}, we must have
$(\pi_1 \otimes \chi^{-1})^\vee \cong  (\pi_1 \otimes \chi^{-1})^\sigma$. By Proposition \ref{mult1},
   if $n$ is an odd integer,
   $(\pi_1 \otimes \chi^{-1})^\vee \cong  (\pi_1 \otimes \chi^{-1})^\sigma$
   if and only if 
$\theta$ restricted to $\F_{q^n}^{\times}$ arises from $\chi$ restricted to
   $\F^\times$ through the norm mapping   $\F_{q^{n}}^\times \rightarrow \F^\times$. In this case,
   $m({\pi}^{\E}, {\chi}^{\E}) \leq 1$, and since by Corollary \ref{mult-cor},
   $ m({\pi}, {\chi}) \leq m({\pi}^{\E}, {\chi}^{\E}) \leq 1$
   with  $ m({\pi}, {\chi}) \cong m({\pi}^{\E}, {\chi}^{\E}) $ modulo 2,
   $ m({\pi}, {\chi}) =  m({\pi}^{\E}, {\chi}^{\E}) $
  completing the proof of the proposition for $n$ odd, in particular, we have $m(\pi_1,\chi) \leq 1$ if $n$ is odd.

If $n$ is an even integer with 
$m({\pi}^{\E}, {\chi}^{\E}) \not = 0$, then since by Proposition \ref{mult1} we must have,
$$(\pi_1 \otimes \chi^{-1})^\vee \not \cong  (\pi_1 \otimes \chi^{-1})^\sigma,$$
by Proposition \ref{princ}, we must have $m(\pi_1, \chi \times \chi^\sigma) \not = 0$ which has been calculated in Theorem  \ref{linear-period}, giving conditions when $m(\pi_1, \chi \times \chi^\sigma) \not = 0$, and proving that
when non-zero, 
$m(\pi_1, \chi \times \chi^\sigma)  = 1$, and therefore by  Proposition \ref{princ}, $m({\pi}^{\E}, {\chi}^{\E})  = 1$, and
so also is  $m({\pi}, {\chi})  = 1$, completing the proof of the theorem.
 \end{proof}

\begin{cor}
  An irreducible cuspidal representation $\pi = \pi(\theta)$ of $\GL_{2n}(\F), n>1$ arising  from a character
  $\theta: \F_{q^{2n}}^\times \rightarrow \C^\times$
  is distinguished by $\GL_n(\E)$ (i.e., has a vector fixed by  $\GL_n(\E)$)
  if and only if one of the equivalent conditions hold:
\begin{enumerate}
\item  $\theta$ restricted to $\F_{q^{n}}^\times$ is trivial,
\item   the representation $\pi$ of $G=\GL_{2n}(\F)$ is self-dual.
\end{enumerate}

Further,  for a character $\chi: \E^\times \rightarrow \C^\times$,
an irreducible cuspidal representation $\pi$ of $\GL_{2n}(\F)$
contains the character  $\chi\circ \det: \GL_n(\E) \rightarrow \C^\times$
if and only if,
$$\pi \cong \pi^\vee \otimes \chi|_{\F^\times}.$$
\end{cor}

\begin{remark}
  It may be noted that Proposition \ref{conj-PT} is false for $n=1$. In this case, $m(\pi^\E,\chi^\E)$ corresponds to the
  multiplicity of a character of the split diagonal torus inside $\GL_2(\E)$ for a principal series representation of $\GL_2(\E)$ which can be $>1$, so the proof of the Proposition \ref{conj-PT} does not apply.
  To elaborate on this, observe that a cuspidal representation
  $\pi = \pi(\theta)$ of $\GL_2(\F)$ has dimension $(q-1)$, whereas there are $(q+1)$ characters of $\E^\times$ whose restriction to
  $\F^\times$ is $\theta|_{\F^\times}$ of which the two characters $\theta$ and $\theta^\sigma$ of $\E^\times$ do not appear
  in $\pi|_{\E^\times}$, and all the other $(q-1)$ appear with multiplicity 1 whereas if Proposition \ref{conj-PT} were valid for
  $n=1$ it would demand all characters of $\E^\times$ whose restriction to
  $\F^\times$ is $\theta|_{\F^\times}$ to appear
  in $\pi|_{\E^\times}$. 
  \end{remark}
\vspace{5mm}

\noindent{\bf Acknowledgement:}
The author must thank N. Matringe for pointing out a mistake in an earlier version
of Proposition \ref{princ}, and thanks the referee for a careful reading of the paper.

\end{document}